\documentclass[11pt]{amsart}

\pdfoutput=1

\usepackage[text={420pt,660pt},centering]{geometry}

\usepackage{cancel}
\usepackage{esint,amssymb}
\usepackage{graphicx}
\usepackage{mathtools} 
\usepackage[colorlinks=true, pdfstartview=FitV, linkcolor=blue, citecolor=blue, urlcolor=blue,pagebackref=false]{hyperref}
\usepackage{microtype}

\usepackage{bm}
\usepackage{dsfont}
\usepackage{mathrsfs}
\usepackage{xcolor}
\usepackage{accents} \usepackage{enumitem} \parskip= 2pt
\usepackage[normalem]{ulem}

\newtheorem{proposition}{Proposition}
\newtheorem{theorem}[proposition]{Theorem}
\newtheorem{lemma}[proposition]{Lemma}

\theoremstyle{remark}
\newtheorem{remark}[proposition]{Remark}

\theoremstyle{definition}

\numberwithin{equation}{section}
\numberwithin{proposition}{section}
\numberwithin{figure}{section}
\numberwithin{table}{section}

\newcommand{\N}{\mathbb{N}}

\newcommand{\R}{\mathbb{R}}

\newcommand{\E}{\mathbb{E}}

\renewcommand{\S}{\mathbf{S}}
\newcommand{\C}{\mathcal{Q}}
\newcommand{\basis}{\Sigma}

\newcommand{\eps}{\varepsilon}

\renewcommand{\leq}{\leqslant}
\renewcommand{\geq}{\geqslant}

\renewcommand{\bar}{\overline}
\renewcommand{\tilde}{\widetilde}

\newcommand{\Ll}{\left}
\newcommand{\Rr}{\right}
\renewcommand{\d}{\mathrm{d}}

\newcommand{\D}{D}

\DeclareMathOperator{\tr}{tr}
\DeclareMathOperator{\supp}{supp}

\newcommand{\la}{\left\langle}
\newcommand{\ra}{\right\rangle}

 \newcommand{\cM}{\mathcal{M}}
\newcommand{\identity}{{\mathbf{Id}}}
\newcommand{\one}{{\boldsymbol{1}}}
\newcommand{\Sym}{{\mathrm{Sym}(\D)}}
\newcommand{\ks}{{\mathbf{s}}}
\newcommand{\fR}{{\mathfrak{R}}}
\newcommand{\bsigma}{{\boldsymbol{\sigma}}}
\newcommand{\balpha}{{\boldsymbol{\alpha}}}
\newcommand{\cleq}{\preccurlyeq}
\newcommand{\cgeq}{\succcurlyeq}
\newcommand{\sP}{\mathscr{P}}
\newcommand{\sF}{\mathscr{F}}
\renewcommand*{\dot}[1]{\accentset{\mbox{\large\bfseries .}}{#1}}
\newcommand{\ellipt}{\mathsf{Ellipt}}

\begin{document}

\author{Hong-Bin Chen}
\address{Institut des Hautes \'Etudes Scientifiques, France}
\email{hbchen@ihes.fr}


\title{On Parisi measures of Potts spin glasses with correction}

\begin{abstract}
In the Potts spin glass model, inspired by the symmetry argument in \cite{bates2023parisi} for the constrained free energy, we study the free energy with self-overlap correction. Similarly, we simplify the Parisi-type formula, originally an infimum over matrix-valued paths, into an optimization over real-valued paths, which can be interpreted as quantile functions of probability measures on the unit interval. Minimizers are usually called Parisi measures. Using results in \cite{chen2023Parisi} generalizing the Auffinger--Chen convexity argument, we deduce the uniqueness of the Parisi measure.

The approach in~\cite{bates2023parisi} is to perturb the covariance function of the Hamiltonian into one associated with the generic model, prove the results there, and then reduce the perturbation. Here, we choose to perturb the model by adding an external field parametrized by Ruelle probability cascades. This is used in the Hamilton--Jacobi equation approach and we directly apply results in~\cite{HJ_critical_pts}.
\end{abstract}

\maketitle

\section{Introduction}

We consider the Potts spin glass model and study the associated free energy with self-overlap correction. In the following, we begin by describing the model in Section~\ref{s.setting}. Subsequently, we present two main results. First, in Section~\ref{s.parisi_simple}, we explain how to leverage the symmetry of the Potts spins to simplify the Parisi-type formula for the limit free energy. The first main result is Theorem~\ref{t.1}, which states that the simplified formula optimizes over real-valued paths. Moving on, in Section~\ref{s.PDE_rep}, we reformulate the formula in terms of solutions to the Parisi PDEs. Once again, the formula for the limit free energy optimizes over real-valued paths, which can be interpreted as cumulative distribution functions of probability measures on the real line. This constitutes a part of the second main result, Theorem~\ref{t.main}, which also establishes the uniqueness of minimizers. The probability measures represented by these minimizers are commonly referred to as the Parisi measures. Consequently, Theorem~\ref{t.main} demonstrates the uniqueness of Parisi measures in the Potts spin glass with self-overlap correction.

\subsection{Setting}\label{s.setting}
Fix an integer $D\geq 2$ and let $\basis = \{e_1,\dots,e_\D\}$ be the standard basis of $\R^\D$. 
We interpret $\basis$ as the set of $D$ Potts spins.
We denote by $P_1$ the uniform probability measure on $\basis$. 
For $N\in\N$, 
let $\Sigma_N$ be the state space of configurations of $N$ spins, which are $\D\times N$ matrices with column vectors in $\Sigma$.
We sample $\sigma$ in $\Sigma_N$ by drawing its column vector $\sigma_{\bullet i}=(\sigma_{ki})_{1\leq k\leq \D}$ independently from $P_1$, for each $i\in\{1,\dots,N\}$. More precisely, denoting the distribution of $\sigma$ by $P_N$, we have $\d P_N(\sigma) = \otimes_{i=1}^N \d P_1(\sigma_{\bullet i})$.

For matrices or vectors $a$ and $b$ of the same dimension, we denote by $a\cdot b$ the entry-wise inner product and write $|a|=\sqrt{a\cdot a}$. We denote by $\S^\D$ the linear space of real symmetric $\D\times\D$ matrices, equipped with the entry-wise inner product in $\R^{\D\times \D}$. Occasionally, we write $\la a,b\ra_{\S^\D} = a\cdot b$ to emphasize the ambient space. We denote by $\S^\D_+$ (resp.\ $\S^\D_{++}$) the subset of $\S^\D$ consisting of positive semi-definite (resp.\ definite) matrices. For $a,b \in \R^{\D\times\D}$, we write $a\cgeq b$ and $b\cleq a$ if $a\cdot c \geq b\cdot c$ for every $c\in\S^\D_+$.

Throughout, we assume that $\xi:\R^{\D\times\D}\to\R$ satisfies the following:
\begin{enumerate}[start=1,label={\rm (H\arabic*)}]
    \item \label{i.xi_loc_lip}
    $\xi$ is differentiable and $\nabla\xi$ is locally Lipschitz;
\item  \label{i.xi_sym}
    $\xi\geq 0$ on $\S^\D_+$, $\xi(0)=0$, and $\xi(a)=\xi(a^\intercal)$ for all $a\in \R^{\D\times\D}$;
    \item \label{i.xi_incre} if $a,\,b\in\S^\D_+$ satisfies $a\cgeq b$, then $\xi(a)\geq \xi(b)$ and $\nabla \xi(a)\cgeq \nabla \xi(b)$.
\end{enumerate}
Here, differentiability and the derivative of $\xi$ are defined with respect to the inner product on $\R^{\D\times \D}$. By~\ref{i.xi_sym} and \ref{i.xi_incre}, we can deduce that $\nabla\xi(s) \in \S^\D_+$ for $a\in\S^\D_+$. 
We further assume that $\xi$ is symmetric in the following sense.
Let $\Sym$ be the permutation group acting on $\{1,\dots,\D\}$. For $a\in\R^{\D\times \D}$ and $\ks \in \Sym$, we write
\begin{align}\label{e.a^s=}
    a^\ks = \Ll(a_{\ks(k),\ks(k')}\Rr)_{1\leq k,k'\leq \D}.
\end{align}
Additionally, we assume
\begin{enumerate}[start=1,label={\rm (H4)}]
\item \label{i.xi(a)=xi(a^s)} $\xi(a) = \xi\Ll(a^\ks\Rr)$ for every $ a \in\R^{\D\times \D}$ and every $ \ks \in \Sym$.
\end{enumerate}

Assume that, for each $N\in\N$, there exists a centered Gaussian process $(H_N(\sigma))_{\sigma\in \Sigma_N}$ with covariance
\begin{align}\label{e.xi_cov}
    \E H_N(\sigma)H_N(\sigma') = N\xi\Ll(\frac{\sigma\sigma'^\intercal}{N}\Rr),\quad\forall \sigma,\sigma'\in\Sigma_N.
\end{align}
We consider
\begin{align}\label{e.F_N=}
    F_N = \frac{1}{N}\E \log\int \exp\Ll(H_N(\sigma) - \frac{N}{2} \xi\Ll(\frac{\sigma\sigma^\intercal}{N}\Rr)\Rr)\d P_N(\sigma).
\end{align}
We call $- \frac{N}{2} \xi\Ll(\frac{\sigma\sigma^\intercal}{N}\Rr)$ the \textit{self-overlap correction}. To stress the difference from the standard definition of free energy without correction, we call $F_N$ the \textit{free energy with self-overlap correction}. Note that the correction is equal to $-\frac{1}{2}\E H_N(\sigma)H_N(\sigma)$ which resembles the drift term in an exponential martingale.

\subsection{Parisi-type formulas and simplification via symmetry of spins}\label{s.parisi_simple}
When $\xi$ is convex, Panchenko proved in \cite{pan.potts} that the standard free energy without correction converges to a generalized Parisi-type variational formula of the form ``$\sup\inf$''. 
Recall that the classical Parisi formula for the Sherrington--Kirkpatrick (SK) model is an infimum. For the Potts model, the additional supremum is needed because the \textit{self-overlap} $\frac{\sigma\sigma}{N}$ is not constant and an extra variation over possible values of the self-overlap is required.
With self-overlap correction, we can bypass this issue and get that, if $\xi$ is convex on $\S^\D_+$, then
\begin{align}\label{e.limF_N=infsP}
    \lim_{N\to\infty} F_N = \inf_{\pi\in\Pi} \sP(\pi)
\end{align}
resembling the classical Parisi formula. This is proved in~\cite[Corollary~8.3]{HJ_critical_pts} (also see \cite[Remark~8.5]{HJ_critical_pts}).
Here, the infimum is taken over
\begin{align}\label{e.Pi}
    \Pi = \Ll\{\pi:[0,1]\to\S^\D_+\,\big|\, \text{$\pi$ is left-continuous and increasing}\Rr\}
\end{align}
where $\pi$ is said to be increasing if $\pi(t)\cgeq \pi(s)$ whenever $t\geq s$.
The definition of the Parisi-type functional $\sP$ involves the Ruelle probability cascade and is postponed to~\eqref{e.sP=}.

Since the limit in~\eqref{e.limF_N=infsP} is an infimum, one can employ existing techniques for the SK model. By adding a deterministic external field of the form $h\cdot \sigma\sigma^\intercal$ to $F_N$ and using methods for the generic mixed $p$-spin models as in \cite[Section~3.7]{pan}, one can show that the self-overlap actually converges. Using the symmetry of $P_1$ under permutation of spin types, one can further identify the limit of self-overlap to be $\frac{1}{\D}\identity_\D$ and deduce that, if $\xi$ is convex on $\S^\D_+$, then
\begin{align}\label{e.limF_N=infsP_fix_endpt}
    \lim_{N\to\infty}F_N = \inf_{\pi\in\Pi(\frac{1}{\D}\identity_\D)} \sP(\pi)
\end{align}
for $\Pi(\frac{1}{\D}\identity_\D) = \Ll\{\pi\in\Pi:\pi(1) = \frac{1}{\D}\identity_\D\Rr\}$. Here and throughout, $\identity_\D$ is the $\D\times\D$ identity matrix. These results are recorded in \cite[Corollary~1.2]{chen2023on}.

Since $\Pi(\frac{1}{\D}\identity_\D)$ still consists of matrix-valued paths,
we want to further exploit the symmetry to refine the formula as in \cite{bates2023parisi} in the context of constrained free energy.
Let us explain this.
Heuristically, if $\pi$ is a minimizing path, then the law of $\pi(U)$ is equal to the limit distribution of the \textit{overlap} $\frac{\sigma\sigma'^\intercal}{N}$, where $U$ is the uniform random variable over $[0,1]$. Since spins take value in $\Sigma$, we always have $\sum_{k,k'=1}^\D \Ll(\frac{\sigma\sigma'^\intercal}{N}\Rr)_{k,k'} =1$. Also, the symmetry of $P_1$ implies $\frac{\sigma\sigma'^\intercal}{N}\stackrel{\d}{=} \Ll(\frac{\sigma\sigma'^\intercal}{N}\Rr)^\ks$ for every $\ks\in\Sym$. Using these two properties, one can deduce that $\pi$ satisfies $\sum_{k,k'}\pi_{k,k'}=1$, that the diagonal entries of $\pi$ are identical, and that the off-diagonal entries are also identical. Hence, $\pi$ takes the form $\pi = \Psi\Ll(\frac{\D\tr(\pi)-1}{\D-1}\Rr)$ where
\begin{align}\label{e.Psi_potts}
    \Psi(s) = \frac{s}{\D}\identity_\D + \frac{1-s}{\D^2}\one_\D,\quad\forall s\in[0,1].
\end{align}
Here and throughout, $\one_\D = (1)_{1\leq i,j\leq \D}\in\R^{\D\times\D}$ has all entries equal to $1$.
This leads to the following result.
\begin{theorem}\label{t.1}
    If $\xi$ is convex on $\S^\D_+$, then
    \begin{align}\label{e.limF_N=inf_zeta}
        \lim_{N\to\infty} F_N = \inf_{\zeta} \sP\Ll(\Psi\circ \zeta\Rr)
    \end{align}
    where the infimum is taken over all left-continuous and increasing $\zeta:[0,1]\to[0,1]$.
\end{theorem}

\subsection{Representation by Parisi PDE solutions and uniqueness of minimizers}\label{s.PDE_rep}
The right-hand side in~\eqref{e.limF_N=inf_zeta} resembles the classical Parisi formula for the SK model where the infimum is also taken over paths in one dimension.
It is well-known \cite{aufche} that the Parisi formula in the SK model has a unique minimizer. Hence, it is natural to investigate the uniqueness of the minimizer in~\eqref{e.limF_N=inf_zeta}.

Before proceeding, we point out a nuance in the interpretations of paths.
For the classical Parisi formula, the one-dimensional path is viewed as the cumulative distribution function (c.d.f.) of a probability measure on $[0,1]$. Here, the correct interpretation of $\zeta$ is the \textit{quantile function} of a probability measure, which is the left-continuous inverse of the c.d.f. 

The uniqueness result in~\cite{aufche} was proved by showing that the Parisi functional is strictly convex over c.d.f.s.
To adopt a similar approach, we need to rewrite~\eqref{e.limF_N=inf_zeta} in terms of c.d.f.s. In the SK model, a natural way to express the Parisi functional in c.d.f.s is through solutions of the Parisi PDE. Since Potts spins are vector-valued, we need to consider generalized Parisi PDEs.

To describe the family of PDEs, we set
\begin{gather}
    \cM = \Ll\{\alpha:[0,1]\to[0,1]\,\big|\,\text{$\alpha$ is increasing and right-continuous; $\alpha(0)=0$; $\alpha(1)=1$}\Rr\};\label{e.cM}
    \\
    \mu(s) = \nabla\xi(\Psi(s)),\quad\forall s\in[0,1]. \label{e.mu=}
\end{gather}
We interpret $\cM$ as the set of c.d.f.s of probability measures on $[0,1]$. 
We write $\dot\mu(s) = \frac{\d}{\d s}\mu(s)$, which exists for almost every $s\in[0,1]$ due to Rademacher's theorem.

For each $\alpha\in\cM$, we consider the backward parabolic PDE
\begin{align}\label{e.Parisi_PDE}
    \partial_s\Phi_\alpha(s,x) + \frac{1}{2}\la \dot\mu(s), \nabla^2\Phi_\alpha(s,x)+\alpha(s)(\nabla\Phi_\alpha(s,x))(\nabla\Phi_\alpha(s,x))^\intercal \ra_{\S^\D}=0
\end{align}
for $(s,x)\in[0,1]\times \R^\D$ with terminal condition
\begin{align}\label{e.terminal_cond}
    \Phi_\alpha(1,x) = \log\int\exp\Ll(x\cdot \sigma-\frac{1}{2}\mu(1)\cdot\sigma\sigma^\intercal\Rr)\d P_1(\sigma).
\end{align}
The notion of solutions and well-posedness will be recalled in Section~\ref{s.pf}. 
We also need the function $\theta:\R^{\D\times\D}\to\R$ given by 
\begin{align}\label{e.theta}
    \theta(a) = a\cdot\nabla\xi(a) -\xi(a),\quad\forall a \in \R^{\D\times \D}.
\end{align}
For $\alpha\in \cM$, we define
\begin{align}\label{e.sF}
    \sF(\alpha) = \E \Ll[\Phi_\alpha\Ll(0,\sqrt{\mu(0)}\eta\Rr)\Rr]+\frac{1}{2}\theta(\Psi(1))-\frac{1}{2}\int_0^1\alpha(s)\Psi(s)\cdot \dot\mu(s)\d s
\end{align}
where $\eta$ is the standard $\R^\D$-valued Gaussian vector.
We will show in Lemma~\ref{l.sF=sP} that $\sF(\alpha) = \sP(\Psi\circ\alpha^{-1})$ if $\alpha$ is a step function, where $\alpha^{-1}$ is the left-continuous inverse of $\alpha$. Our main result is stated below.

\begin{theorem}\label{t.main}
Suppose that 
$\xi$ satisfies either one of the following:
\begin{enumerate}
    \item \label{i.t.main_1} $\dot\mu$ exists everywhere on $[0,1]$ and $\dot\mu(s)\in\S^\D_{++}$ for almost every $s\in[0,1]$ with respect to the Lebesgue measure;
    \item \label{i.t.main_2} $\xi(a) =\beta^2|a|^2$ for all $a\in\R^{\D\times \D}$ and some $\beta>0$.
\end{enumerate}
If $\xi$ is convex on $\S^\D_+$, then
\begin{align}\label{e.limF_N=infF}
    \lim_{N\to\infty} F_N = \inf_{\alpha\in\cM}\sF(\alpha).
\end{align}
Moreover, $\cM\ni\alpha\mapsto \sF(\alpha)$ is strictly convex and $\inf_{\alpha\in\cM}\sF(\alpha)$ has a unique minimizer.
\end{theorem}

The strict convexity is understood in the following sense: for distinct $\alpha_0,\alpha_1\in\cM$, we have $(1-\lambda)\sF(\alpha_0)+\lambda\sF(\alpha_1)< \sF((1-\lambda)\alpha_0+\lambda \alpha_1)$ for all $\lambda \in(0,1)$.

\begin{remark}
When $\xi$ satisfies~\eqref{i.t.main_2}, we simply have $\mu =2\beta \Psi$ and $\dot\mu= 2\beta(\frac{1}{\D}\identity_\D-\frac{1}{\D^2}\one_\D)$.
Notice that $\D\identity_\D -\one_\D$ is not positive definite (the vector with equal entries is an eigenvector with the zero eigenvalue; see \cite[(5.12) and (5.13)]{chen2023Parisi}).
Hence, \eqref{i.t.main_1} and \eqref{i.t.main_2} have no overlap.
\end{remark}

\begin{remark}The two conditions~\eqref{i.t.main_1} and~\eqref{i.t.main_2} cover the case of symmetric mixed-$p$ interactions. More precisely, consider $\xi$ given by $\xi(a) = \sum_{p=2}^\infty \beta_p^2 \sum_{k,k'=1}^\D a_{k,k'}^p$ for all $a\in\R^{\D\times\D}$ where $(\beta_p)_{p=2}^\infty$ is a sequence of nonnegative numbers decaying sufficiently fast. The process $(H_N(\sigma))_{\sigma\in\Sigma_N}$ can be constructed explicitly (for instance, see the beginning of \cite[Section~5]{chen2023Parisi}). If $\beta_p>0$ for some $p\geq 3$, then $\xi$ satisfies~\eqref{i.t.main_1}, which is verified in \cite[Lemma~5.2]{chen2023Parisi}. If $\beta_2>0$ and $\beta_p=0$ for all $\beta\geq 3$, then $\xi$ is of the form in~\eqref{i.t.main_2}. 
\end{remark}

\subsection{Parisi measures}

In the SK model, the Parisi formula is an infimum over the collection $\cM$.
A minimizer of the Parisi formula, viewed as a probability measure on $[0,1]$, is called a \textit{Parisi measure}. In the Potts spin glass, we can interpret the minimizers of $\inf_{\alpha\in\cM} \sF(\alpha)$ as Parisi measures by identifying $\alpha$ with the measure it presents. Theorem~\ref{t.main} implies the uniqueness of the Parisi measure. 
However, there are notable limitations to this result.

The uniqueness is only proven after we refined the variational formula into~\eqref{e.limF_N=infF}. We do not know the uniqueness of minimizers in~\eqref{e.limF_N=infsP} or~\eqref{e.limF_N=infsP_fix_endpt} (where the law of $\pi(U)$ for a minimizer $\pi$ can also be a good definition of Parisi measures which live on $\S^\D_+$).
Also, to study further properties of the Parisi measure, it would be beneficial if the Parisi measure could be the unique minimizer over larger collections.

Also, our results require the self-overlap correction. There is a way to remove the correction as in \cite[Section~5]{mourrat2020extending} and \cite[Section~5]{chen2023self}, but this needs to add external fields of the form $h\cdot\sigma\sigma^\intercal$ for $h\in \S^\D$. Since the reduction to~\eqref{e.limF_N=infF} relies on the symmetry of $P_1$ and since the external field in general breaks this symmetry, we do not know how to extend our results to the standard setting without the correction. 

Nevertheless, we believe that the uniqueness of the currently defined Parisi measure could allow further investigations into the Potts spin glass.

\subsection{Related works}

The Potts spin glass was introduced by Elderfield and Sherrington in~\cite{elderfield1983curious} and has since been extensively studied in physics \cite{caltagirone2012dynamical, de1995static, elderfield1983novel, gross1985mean, marinari1999glassy, nishimori1983gauge}.

For the SK model, the Parisi formula was initially proposed in \cite{parisi79,parisi80} and later mathematically established in \cite{gue03,Tpaper}.
Generalizations were made in various settings \cite{pan05,tal.sph,panchenko2014parisi,chen2013aizenman,pan.multi,bates2022free,pan.potts,pan.vec,mourrat2020extending}. Among them, the free energy of the Potts spin glass with quadratic interaction was identified in \cite{pan.potts} and general interactions in~\cite{pan.vec}.

The self-overlap correction appeared in the Hamilton--Jacobi equation approach by Mourrat \cite{mourrat2022parisi,mourrat2020extending,mourrat2021nonconvex,mourrat2023free}. Spin glass models with correction have been recently considered in \cite{chen2022pde,chen2023self,chen2023on,chen2023Parisi}. For more detail on the Hamilton--Jacobi equation approach to statistical mechanics, we refer to~\cite{HJbook}.

Recently, Bates and Sohn considered the constrained free energy in the Potts spin glass~\cite{bates2023parisi}. When the constraint is balanced, they proved that the Parisi formula for the constraint free energy optimizes over paths with values in $\Psi([0,1])$. Their method inspired this work. Theorem~\ref{t.1} is a version of their result for the free energy with self-overlap correction. To make the heuristic derivation of~\eqref{e.Psi_potts} rigorous, \cite{bates2023parisi} introduces a perturbation $\tilde\xi$ of $\xi$ so that $\tilde\xi$ is generic in the sense that it includes sufficiently many interactions of various orders. By differentiating the limit free energy associated with $\tilde\xi$ in terms of the inverse temperature for each interaction, one can obtain information to determine the limit distribution of the overlap, which minimizes the Parisi formula. When $\xi$ is symmetric in the sense of~\ref{i.xi(a)=xi(a^s)}, $\tilde\xi$ can be chosen to be symmetric. Varying the action of $\Sym$ on the spin types, the symmetry of $\tilde\xi$ implies that the limit overlap distribution is supported on $\Psi([0,1])$. Finally, the continuity of the free energy and the Parisi formula in terms of $\xi$ yields the result for $\xi$.

Our strategy for Theorem~\ref{t.1} is essentially the same, but we use a different kind of perturbation that appears in the Hamilton--Jacobi approach. We keep $\xi$ intact and add an external field parameterized by a Ruelle probability cascade. We use the differentiability of the Parisi formula with respect to the cascade parameter to gain information on the limit distribution of the overlap. We directly cite relevant results proved in \cite{HJ_critical_pts} to avoid repeating technical details. By choosing the external field to be symmetric, we can conclude Theorem~\ref{t.1} through similar arguments via symmetry and continuity as above.

The proof of Theorem~\ref{t.main} requires results on the Parisi PDE.
In the setting of the SK model, the Parisi PDE has been widely studied in \cite{talagrand2006parisi,auffinger2015properties,aufche,talagrand2006parisi}. 
The uniqueness of the Parisi measure was predicted in \cite{mezard1987spin}. The strict convexity of the Parisi functional was investigated in \cite{panchenko2005question,panchenko2005question,talagrand2006parisi,Tpaper,bovklim,chen2015partial,jagtob} and was proved in \cite{aufche}.
The multi-dimensional Parisi PDEs were considered in \cite{bovklim} in the setting of vector spin glasses with quadratic interaction. For general interactions, such Parisi PDEs were recently revisited in \cite{chen2023Parisi}.

\subsection{Acknowledgements}
The author is grateful to Erik Bates from whom the author learned the symmetry argument appearing in \cite{bates2023parisi}. This project has received funding from the European Research Council (ERC) under the European Union’s Horizon 2020 research and innovation programme (grant agreement No.\ 757296).

\section{Enriched free energy}\label{s.enriched_free_energy}

To make the argument leading towards~\eqref{e.Psi_potts} rigorous, we need to consider a perturbed version of $F_N$. In short, one needs the Hamiltonian to be ``rich'' enough to determine the limit overlap. Here, we use the enrichment in the Hamilton--Jacobi equation approach.

We recall the setting in~\cite{HJ_critical_pts} and some useful results.
Set
\begin{align*}
    \C = \Ll\{q:[0,1)\to\S^\D_+\,\big|\, t\geq s \implies q(t)\cgeq q(s)\Rr\}.
\end{align*}
We identify elements in $\C$ if they coincide almost everywhere on $[0,1)$ with respect to the Lebesgue measure. 
(In~\cite[(1.4)]{HJ_critical_pts}, in addition to monotonicity, $q\in\C$ is required to be càdlàg. But due to the monotonicity and the aforementioned identification, this is not important.)
Let $L^\infty$ be the collection of maps $\kappa:[0,1)\to\S^\D$ such that $\|\kappa\|_{L^\infty} = \mathrm{ess\,sup}_{s\in[0,1]}|\kappa(s)|<\infty$. For brevity, we write $\C_\infty = \C\cap L^\infty$.
For every $q\in \C_\infty$, we always set $q(1) = \lim_{s\nearrow 1}q(s)$ which exists due to the monotonicity. 

Let $\fR$ be the Ruelle probability cascade on some separable Hilbert space $\mathfrak{H}$ with overlap distributed uniformly over $[0,1]$. More precisely (see in \cite[Section~2.3]{pan} or \cite[Section~5.6]{HJbook}), $\fR$ is a random probability measure on the unit sphere of $\mathfrak{H}$ satisfying the following. Denoting the inner product on $\mathfrak{H}$ by $\alpha \wedge\alpha'$ and sampling $(\alpha^l)_{l\in\N}$ independently from $\mathfrak{R}^{\otimes\infty}$, we have that $(\alpha^l\wedge\alpha^{l'})_{l,l'\in\N}$ satisfies the Ghirlanda--Guerra identities under $\E\mathfrak{R}^{\otimes\infty}$ (whose exact expressions are not needed and we refer to \cite[(2.106)]{pan}). Moreover, the overlap $\alpha^1\wedge\alpha^2$ distributes uniformly over $[0,1]$ under $\E\fR^{\otimes 2}$. The existence of $\fR$ is given by \cite[Theorem~2.17]{pan}.

For almost every realization of $\fR$ and every $q\in\C_\infty$, there exists an $\R^\D$-valued centered Gaussian process $(w^q(\alpha))_{\alpha \in\supp\fR}$ with covariance
\begin{align}\label{e.w^q}
    \E w^q(\alpha)w^q(\alpha')^\intercal = q\Ll(\alpha\wedge\alpha'\Rr),\quad\forall \alpha,\alpha'\in \supp\fR.
\end{align}
We refer to \cite[Section~4]{HJ_critical_pts} for the construction and measurability of this process.

Throughout, we take $\R_+=[0,\infty)$. Recall the original Hamiltonian $H_N(\sigma)$ described in~\eqref{e.xi_cov}. For every $N\in\N$ and every $(t,q)\in \R_+\times \C_\infty$, we define, for $(\sigma,\alpha)\in \Sigma_N\times \supp\fR$,
\begin{align}\label{e.H^t,q}
    H^{t,q}_N(\sigma,\alpha) = \sqrt{2t}H_N(\sigma) - tN\xi\Ll(\frac{\sigma\sigma^\intercal}{N}\Rr)+ \sqrt{2}W_N^q(\alpha)\cdot \sigma - q(1)\cdot \sigma\sigma^\intercal
\end{align}
where $W^q_N(\alpha)$ is an $\R^{\D\times N}$-valued process consisting of column vectors $(w_i^q(\alpha))_{i=1}^N$ that are independent copies of $w^q(\alpha)$ (conditioned on $\fR$).
We take $W^q_N(\alpha)$ to be independent from $H_N(\sigma)$.
We can view $H^{\frac{1}{2},q}_N(\sigma,\alpha)$ as $H_N(\sigma)$ enriched (or perturbed) by an external field. The Gibbs measure corresponding to $H^{t,q}_N(\sigma,\alpha)$ is 
\begin{align*}
    \la\cdot\ra_{N,t,q} \propto \exp\Ll(H^{t,q}_N(\sigma,\alpha)\Rr)\d P_N(\sigma)\d \mathfrak{R}(\alpha).
\end{align*}
We denote by $(\sigma,\alpha)$ the canonical random variable under $\la\cdot\ra_{N,t,q}$. We denote the tensorization of $\la\cdot\ra_{N,t,q}$ still by $\la\cdot\ra_{N,t,q}$. An independent copy of $(\sigma,\alpha)$ under $\la\cdot\ra_{N,t,q}$ is written as $(\sigma',\alpha')$.

We consider the enriched free energy, for $(t,q)\in \R_+\times \C_\infty$,
\begin{align}\label{e.f_N}
    f_N(t,q) = -\frac{1}{N}\E\log \iint \exp\Ll(H^{t,q}_N(\sigma,\alpha)\Rr)\d P_N(\sigma)\d \mathfrak{R}(\alpha)
\end{align}
(notice the minus sign), which is the free energy defined in \cite[(5.2)]{HJ_critical_pts} (denoted by $\bar F_N$ therein).
Note that $F_N$ in~\eqref{e.F_N=} satisfies $F_N = -f_N(\frac{1}{2},0)$. We can view $-f_N(\frac{1}{2},q)$ as $F_N$ with an external field.
For $q\in \C_\infty$, we also set $\psi(q) = f_1(0,q)$ which has the expression
\begin{align}\label{e.psi=}
    \psi(q) = -\E \log \iint\exp\Ll(\sqrt{2}w^q(\alpha)\cdot \tau - q(1)\cdot \tau\tau^\intercal\Rr)\d P_1(\tau) \d \fR(\alpha).
\end{align}
We state the Parisi formula for $f_N(t,q)$, which is borrowed from \cite[Proposition~8.1]{HJ_critical_pts}. Recall $\theta$ from~\eqref{e.theta}.

\begin{proposition}[\cite{HJ_critical_pts}]\label{p.f(t,q)=}
Assume that $\xi$ is convex on $\S^\D_+$.
For every $(t,q)\in \R_+\times \C_\infty$, as $N\to\infty$, $f_N(t,q)$ converges to
\begin{align*}f(t,q) = \sup_{p\in \C_\infty} \Ll\{\psi\Ll(q + t\nabla\xi\circ p\Rr) - t\int_0^1 \theta\Ll(p(s)\Rr)\d s\Rr\}.
\end{align*}
\end{proposition}

We will need the following Lipschitzness, which is \cite[Proposition~3.1]{mourrat2023free} (and restated as \cite[Proposition~5.1]{HJ_critical_pts}).

\begin{lemma}[\cite{mourrat2023free}]For every $N\in\N$ and every $t\in\R_+$,
\begin{align*}
    \Ll|f_N(t,q) -f_N(t,q')\Rr|\leq \int_0^1\Ll|q(s)-q'(s)\Rr|\d s,\quad\forall q,\,q'\in\C_\infty.
\end{align*}
\end{lemma}

As a result, we also have, for every $t\in\R_+$,
\begin{gather}\label{e.lip_psi_f}
    \Ll|\psi(q) -\psi(q')\Rr|,\ \Ll|f(t,q) -f(t,q')\Rr|\leq \int_0^1\Ll|q(s)-q'(s)\Rr|\d s,\quad\forall q,q'\in\C_\infty.
\end{gather}
Similar to the generic model, the differentiability of the limit free energy allows us to gain information on the limit distribution of the overlap. The following proposition can be extracted from \cite[Corollary~8.7]{HJ_critical_pts}. To state it, we need more notation. Set $L^\infty_{\leq 1} = \{\kappa\in L^\infty:\:\|\kappa\|_{L^\infty}\leq 1\}$.
For $a\in\S^\D_+$, we define $\ellipt(a)$ to be the ratio of the largest eigenvalue of $a$ over its smallest eigenvalue. We set $\C_\uparrow$ to be
\begin{align*}
    \Ll\{q\in\C\,\big|\, \exists c>0:\:  t\geq s \implies \Ll(q(t)-q(s)\cgeq c(t-s)\identity_\D \text{ and } \ellipt(q(t)-q(s))\leq c^{-1}\Rr)\Rr\}.
\end{align*}

\begin{proposition}[\cite{HJ_critical_pts}]\label{p.p}
Assume that $\xi$ is convex on $\S^\D_+$.
For every $(t,q)\in \R_+\times (\C_\uparrow\cap L^\infty)$, there exists $p \in \C\cap L^\infty_{\leq 1}$ such that
\begin{align}\label{e.f(t,q)=...p...}
    f(t,q) = \psi\Ll(q + t\nabla\xi\circ p\Rr) - t\int_0^1 \theta (p(s))\d s.
\end{align}
Moreover, there is a sequence $(p_N)_{N\in\N}$ in $\C_\uparrow\cap L^\infty_{\leq 1}$ such that
\begin{align}\label{e.p_N=}
    p_N\Ll(\alpha\wedge\alpha'\Rr) = \E \la \frac{\sigma\sigma'^\intercal}{N}\Big| \alpha\wedge\alpha'\ra_{N,t,q},
\end{align}
almost surely under $\E\la\cdot\ra_{N,t,q}$, and $(p_N)_{N\in\N}$ converges to $p$ in $L^1([0,1];\S^\D)$.
\end{proposition}

In~\eqref{e.p_N=}, the conditional expectation is taken with respect to the measure $\E\la\cdot\ra_{N,t,q}$ (not $\la\cdot\ra_{N,t,q}$).

We also need the invariance of the Ruelle probability cascade, which implies that $\alpha\wedge\alpha'$ has the same distribution under $\E\la\cdot\ra_{N,t,q}$ and $\E\la\cdot\ra_\fR$. Hence and henceforth, we write $\la\cdot\ra_\fR = \fR^{\otimes\infty}$. This property is well-known for discrete Ruelle probability cascades (for instance, see \cite[Theorem~4.4]{pan}). For the continuous cascade considered here, this invariance is proved in \cite[Proposition~4.8]{HJ_critical_pts}. We restate it below.
\begin{lemma}[\cite{HJ_critical_pts}]\label{l.inv}
For every $(t,q)\in \R_+\times \C_\infty$ and every bounded measurable function $h:\R\to\R$, it holds that
\begin{align*}
    \E \la h\Ll(\alpha\wedge\alpha'\Rr)\ra_{N,t,q}=\E \la h\Ll(\alpha\wedge\alpha'\Rr)\ra_\fR = \int_0^1 h(s)\d s.
\end{align*}
\end{lemma}

\section{Consequences of the symmetry}

For $\ks\in\Sym$, we denote its inverse by $\ks^{-1}$.
Recall the meaning of $a^\ks$ in~\eqref{e.a^s=}. Notice
\begin{align}\label{e.a.b^s=}
    a\cdot b^\ks = a^{\ks^{-1}}\cdot b,\quad\forall a,b \in \R^{\D\times \D},\ \forall \ks \in \Sym.
\end{align}
For any spin configuration $\sigma\in\Sigma_N\in\R^{\D\times N}$, we write $ \sigma^\ks = (\sigma_{\ks(k),i})_{1\leq d\leq \D,\, 1\leq i\leq N}$. We have
\begin{align}\label{e.sigmasigma^s=}
    \Ll(\sigma\sigma'^\intercal\Rr)^\ks = \sigma^\ks \Ll(\Ll(\sigma'\Rr)^\ks\Rr)^\intercal,\quad\forall \sigma,\sigma'\in \Sigma_N.
\end{align}
Due to $P_1$ being uniform on $\Sigma$,
for every bounded measurable function $h:\R^{\D\times N}\to\R$, we have
\begin{align}\label{e.int_h_dP_N}
    \int h\Ll(\sigma^\ks\Rr) \d P_N(\sigma) = \int h(\sigma) \d P_N(\sigma),\quad\forall \ks \in \Sym.
\end{align}
We call a function $\kappa:[0,1]\to \R^{\D\times\D}$ \textit{$\Sym$-invariant} if $\kappa^\ks(s) =\kappa(s)$ for every $\ks \in \Sym$ and almost every $s\in[0,1]$.

\begin{lemma}\label{l.Gibbs_inv}
If $q\in \C_\infty$ is $\Sym$-invariant, then, for every $N\in\N$, every $t\geq 0$, and every bounded measurable function $h:\Sigma_N^n\times \mathfrak{H}^n\to\R$ with some $n\in\N$, it holds that
\begin{align*}
    \E \la h\Ll(\Ll(\sigma^{\leq n}\Rr)^\ks , \alpha^{\leq n}\Rr) \ra_{N,t,q} = \E \la h\Ll(\Ll(\sigma^{\leq n}\Rr) , \alpha^{\leq n}\Rr) \ra_{N,t,q}
\end{align*}
where $\Ll(\sigma^{\leq n}\Rr)^\ks = \Ll(\Ll(\sigma^l\Rr)^\ks\Rr)_{l=1}^n$, $\sigma^{\leq n} = \Ll(\sigma^l\Rr)_{l=1}^n$, and $\alpha^{\leq n}\ = \Ll(\alpha^l\Rr)_{l=1}^n$ for each $(\sigma^l,\alpha^l)$ independently sampled from $\la\cdot\ra_{N,t,q}$.
\end{lemma}
\begin{proof}
Let us write $\bsigma = \sigma^{\leq n}$, $\balpha = \alpha^{\leq n}$, $H\Ll(\bsigma,\balpha\Rr) = \sum_{l=1}^n H^{t,q}_N\Ll(\sigma^l,\alpha^l\Rr)$, and $G=P_N^{\otimes n}\otimes \fR^{\otimes n}$. Using \eqref{e.xi_cov}, \eqref{e.w^q}, and~\eqref{e.H^t,q}, the Gaussian part of $H\Ll(\bsigma,\balpha\Rr)$ has covariance
\begin{align*}
    C\Ll(\Ll(\bsigma,\balpha\Rr),\, \Ll(\bar\bsigma,\bar\balpha\Rr)\Rr)
    =\sum_{l,l'=1}^n 2tN\xi\Ll(\frac{\sigma^l\Ll(\bar\sigma^{l'}\Rr)^\intercal}{N}\Rr) + 2q\Ll(\alpha^l\wedge \bar \alpha^{l'}\Rr)\cdot \sigma^l\Ll(\bar\sigma^{l'}\Rr)^\intercal
\end{align*}
for all $\Ll(\bsigma,\balpha\Rr),\, \Ll(\bar\bsigma,\bar\balpha\Rr)\in \Sigma_N^n\times \Ll(\supp\fR\Rr)^n$.
Hence, due to~\ref{i.xi(a)=xi(a^s)}, \eqref{e.a.b^s=},~\eqref{e.sigmasigma^s=}, and the invariance of $q$, we have
\begin{align*}
    C\Ll(\Ll(\bsigma^\ks,\balpha\Rr),\, \Ll(\Ll(\bar\bsigma\Rr)^\ks,\bar\balpha\Rr)\Rr) = C\Ll(\Ll(\bsigma,\balpha\Rr),\, \Ll(\bar\bsigma,\bar\balpha\Rr)\Rr)
\end{align*}
for every $\ks\in \Sym$.
Since the deterministic part of $H\Ll(\bsigma,\balpha\Rr)$ is given by $-\frac{1}{2}$ times the variance, we deduce that, conditioned on $\fR$,
\begin{align}\label{e.H^s=H}
    \Ll(H\Ll(\bsigma^\ks,\balpha\Rr)\Rr)_{\Ll(\bsigma,\balpha\Rr) \in \Sigma_N^n\times \Ll(\supp\fR\Rr)^n}
    \stackrel{\d}{=}
    \Ll(H\Ll(\bsigma,\balpha\Rr)\Rr)_{\Ll(\bsigma,\balpha\Rr) \in \Sigma_N^n\times \Ll(\supp\fR\Rr)^n}
\end{align}
for every $\ks\in \Sym$.
Then,
\begin{align*}
    \E \frac{\int h(\bsigma^\ks,\balpha)\exp\Ll(H(\bsigma,\balpha)\Rr)\d G(\bsigma,\balpha)}{\int \exp\Ll(H(\bsigma,\balpha)\Rr)\d G(\bsigma,\balpha)} 
    &\stackrel{\eqref{e.int_h_dP_N}}{=}
    \E \frac{\int h(\bsigma,\balpha)\exp\Ll(H\Ll(\bsigma^{\ks^{-1}},\balpha\Rr)\Rr)\d G(\bsigma,\balpha)}{\int \exp\Ll(H\Ll(\bsigma^{\ks^{-1}},\balpha\Rr)\Rr)\d G(\bsigma,\balpha)}
    \\
    & \stackrel{\eqref{e.H^s=H}}{=}
    \E \frac{\int h(\bsigma,\balpha)\exp\Ll(H(\bsigma,\balpha)\Rr)\d G(\bsigma,\balpha)}{\int \exp\Ll(H(\bsigma,\balpha)\Rr)\d G(\bsigma,\balpha)}
\end{align*}
which is the announced identity.
\end{proof}

\begin{lemma}\label{l.p_inv}
For $(t,q) \in \R_+\times (\C_\uparrow\cap L^\infty)$, let $p\in\C\cap L^\infty_{\leq 1}$ be given by Proposition~\ref{p.p}.
If $q$ is $\Sym$-invariant, then so is $p$. Moreover, almost everywhere on $[0,1]$,
\begin{align}\label{e.p=Psi()}
    p = \Psi\Ll(\frac{D\tr(p)-1}{D-1}\Rr).
\end{align}
\end{lemma}

\begin{proof}
Let $(p_N)_{N\in\N}$ be as in Proposition~\ref{p.p}.
Let $\kappa:[0,1]\to \S^\D$ be bounded and measurable. Assuming that $q$ is $\Sym$-invariant, we can compute, for every $N\in\N$ and $\ks\in\Sym$,
\begin{align*}
    \int_0^1 p_N^{\ks}(s)\cdot \kappa(s)\d s \stackrel{\eqref{e.a.b^s=}}{=} \int_0^1 p_N(s) \cdot \kappa^{\ks^{-1}}(s)\d s \stackrel{\text{L.}\ref{l.inv};\eqref{e.p_N=}}{=} \E \la \frac{\sigma\sigma'^\intercal}{N} \cdot \kappa^{\ks^{-1}} \Ll(\alpha\wedge\alpha'\Rr)\ra_{N,t,q}
    \\
    \stackrel{\eqref{e.a.b^s=};\eqref{e.sigmasigma^s=}}{=}\E \la \frac{\sigma^\ks\Ll((\sigma')^\ks\Rr)^\intercal}{N} \cdot \kappa \Ll(\alpha\wedge\alpha'\Rr)\ra_{N,t,q} \stackrel{\text{L.\ref{l.Gibbs_inv}}}{=} \E \la \frac{\sigma\sigma'^\intercal}{N} \cdot \kappa \Ll(\alpha\wedge\alpha'\Rr)\ra_{N,t,q} 
    \\
    \stackrel{\text{L.}\ref{l.inv};\eqref{e.p_N=}}{=} \int_0^1p_N(s)\cdot\kappa(s)\d s.
\end{align*}
Here, ``L.X'' stands for ``Lemma~X''.
Varying $\kappa$, we deduce $p_N^\ks = p_N$. Sending $N\to\infty$, by the convergence in Proposition~\ref{p.p}, we have $p^\ks=p$ and thus $p$ is $\Sym$-invariant.

Next, we derive~\eqref{e.p=Psi()}. Since the column vectors of $\sigma$ belong to $\Sigma$, the entries in each of them sum to $1$ and thus
\begin{align*}
    \sum_{k,k'=1}^\D \Ll(\sigma\sigma'^\intercal\Rr)_{k,k'} = \sum_{j=1}^N\sum_{k,k'=1}^\D\sigma_{k,j} \sigma_{k',j} = N.
\end{align*}
Therefore, in view of~\eqref{e.p_N=}, we have $\sum_{k,k'=1}^\D (p_N)_{k,k'}(\alpha\wedge\alpha')=1$ almost surely under $\E\la\cdot\ra_{N,t,q}$. Then, Lemma~\ref{l.inv} implies $\sum_{k,k'=1}^\D (p_N)_{k,k'}=1$ almost everywhere on $[0,1]$. The convergence of $(p_N)_{N\in\N}$ yields
\begin{align}\label{e.sumkk'p_k,k'=1}
    \sum_{k,k'=1}^\D p_{k,k'}=1
\end{align}
almost everywhere. Since the entries in $\sigma\sigma'^\intercal $ are nonnegative, by \eqref{e.p_N=}, Lemma~\ref{l.inv}, and the convergence, we also have that the entries of $p(s)$ are nonnegative for almost every $s\in[0,1]$.

For distinct $k,k',k'',k'''\in\{1,\dots,\D\}$, we consider permutations $(kk')$ and $(kk'')(k'k''')$ in the cyclic notation. We can see that the $\Sym$-invariance of $p$ implies $p_{k,k}=p_{k',k'}$ and $p_{k,k'}=p_{k'',k'''}$ almost everywhere. Hence, all diagonal entries of $p$ are identical and all off-diagonal entries are identical. Using this and \eqref{e.sumkk'p_k,k'=1}, we get
$p_{k,k} = \frac{1}{\D}\tr(p)$ for every $k$ and $p_{k,k'}= \frac{1-\tr(p)}{\D(\D-1)}$ for distinct $k,k'$.
Therefore $p= \frac{\D \tr(p)-1}{\D(\D-1)}\identity_\D + \frac{1-\tr(p)}{\D(\D-1)}\one_\D$. 
Recall $\Psi$ in~\eqref{e.Psi_potts} and notice that its domain is restricted to $[0,1]$. Let us lift this restriction for now and assume that $\Psi$ is defined as in~\eqref{e.Psi_potts} for every $s\in\R$. Then, we have $p = \Psi\Ll(\frac{D\tr(p)-1}{D-1}\Rr)$ and we want to verify $\frac{D\tr(p)-1}{D-1}\in[0,1]$. Since $p$ is $\S^\D_+$-valued, in view of the expression of $\Psi$, we must have $\frac{D\tr(p)-1}{D-1}\geq 0$. Since the entries in $p$ are nonnegative, we have $\tr(p)\leq \sum_{k,k'=1}^\D p_{k,k'}=1$ and thus $\frac{D\tr(p)-1}{D-1}\leq 1$. This completes our verification of \eqref{e.p=Psi()}.
\end{proof}

To close this section, we show that the variational formula for $f(t,0)$ can be simplified.
\begin{proposition}\label{p.f(t,0)=sup}
For every $t\geq 0$, 
\begin{align*}
    f(t,0) = \sup_{\zeta} \Ll\{\psi\Ll(t\nabla\xi\circ \Psi\circ\zeta\Rr)- t\int_0^1\theta(\Psi(\zeta(s)))\d s\Rr\} 
\end{align*}
where the infimum is taking over all left-continuous and increasing $\zeta:[0,1]\to[0,1]$.
\end{proposition}

For $a,b\in\R$ and $\eps>0$, we write
\begin{align}\label{e.approx_eps}
    a\approx_\eps b \quad \Longleftrightarrow\quad |a-b|\leq \eps . 
\end{align}

\begin{proof}
For each $K\in\N$, we set $q_K(s) = \frac{s}{K}\identity_\D$ for $s\in[0,1]$ and we have $q_K \in \C_\uparrow\cap L^\infty$. Let $p_K$ be given by Proposition~\ref{p.p} associated with $(t,q_K)$. 
Setting $\zeta_K = \frac{\D \tr(p_K)-1}{D-1}$ and using Lemma~\ref{l.p_inv}, we have $p_K = \Psi\circ \zeta_K$. Since we identify paths that coincide almost everywhere and $\zeta_K$ is clearly increasing, we can let $\zeta_K$ be the left-continuous version. Also, we have verified in the above proof that $0\leq \zeta_K\leq 1$. 
Then, for sufficiently large $K$,
\begin{align*}
    f(t,0)\stackrel{\eqref{e.lip_psi_f}}{\approx_\eps} f(t,q_K)  \stackrel{\eqref{e.f(t,q)=...p...}}{=} \psi\Ll(q_K+ t\nabla\xi\circ \Psi\circ\zeta_K\Rr)- t\int_0^1\theta\Ll(\Psi\Ll(\zeta_K(s)\Rr)\Rr)\d s
    \\
    \stackrel{\eqref{e.lip_psi_f}}{\approx_\eps} \psi\Ll(t\nabla\xi\circ \Psi\circ\zeta_K\Rr)- t\int_0^1\theta\Ll(\Psi\Ll(\zeta_K(s)\Rr)\Rr)\d s.
\end{align*}
Given Proposition~\ref{p.f(t,q)=}, the above along with the observation $\Psi\circ\zeta_K\in \C_\infty$ implies the announced identity.
\end{proof}

\section{Proofs of main results}\label{s.pf}

We first prove Theorem~\ref{t.1}. We start by defining the functional $\sP$ using the Ruelle probability cascade $\fR$ introduced in Section~\ref{s.enriched_free_energy}.
Recall $\Pi$ in~\eqref{e.Pi}.
For each $\pi\in\Pi$, set
\begin{align}\label{e.sP=}
    \sP(\pi) = \E \log \iint \exp\Ll(w^{\frac{1}{2}\nabla\xi\circ \pi}(\alpha)\cdot \tau - \frac{1}{2}\nabla\xi\circ\pi(1)\cdot\tau\tau^\intercal\Rr)\d P_1(\tau) \d \fR(\alpha) 
    \\
    + \frac{1}{2}\int_0^1\theta(\pi(s))\d s. \notag
\end{align}

\begin{proof}[Proof of Theorem~\ref{t.1}]
Comparing the definition of $F_N$ in~\eqref{e.F_N=} with that of $f_N$ in~\eqref{e.f_N}, we have $F_N = -f_N(\frac{1}{2},0)$. Comparing the definition of $\psi$ in~\eqref{e.psi=} and that of $\sP$ in~\eqref{e.sP=}, we have
\begin{align}\label{e.sP=-psi+...}
    \sP(\pi) = -\psi\Ll(\frac{1}{2}\nabla\xi\circ \pi\Rr)+ \frac{1}{2}\int_0^1\theta(\pi(s))\d s.
\end{align}
Combining these relations with Proposition~\ref{p.f(t,0)=sup}, we can obtain the desired result. 
\end{proof}

Next, we turn to the proof of Theorem~\ref{t.main}.
We start by recalling the results on the Parisi PDE~\eqref{e.Parisi_PDE} with terminal condition~\eqref{e.terminal_cond}.
Let $\cM_\d$ denote the subcollection of $\cM$ in~\eqref{e.cM} consisting of step functions with finite steps.
Recall $\mu$ defined in~\eqref{e.mu=}.
Then, under the condition
\begin{align}\label{e.mu_cond}
    \text{$ \frac{\d}{\d s}\sqrt{\mu(t)-\mu(s)} $ exists for every $s,t$ satisfying $0\leq s<t\leq 1$,}
\end{align}
we can solve the PDE~\eqref{e.Parisi_PDE} by the Cope--Hopf transformation. More precisely, we can construct explicitly a continuous function $\Phi_\alpha:[0,1]\times \R^\D\to\R$ as in \cite[Definition~2.3]{chen2023Parisi} (with $e^{-\frac{1}{2}\mu(1)\cdot\sigma\sigma^\intercal}\d P_1(\sigma)$ substituted for $P_1$ therein) such that $\Phi_\alpha$ is differentiable and satisfies \eqref{e.Parisi_PDE} at every continuity point of $\alpha$, which is verified in \cite[Lemma~2.4]{chen2023Parisi}. 
By construction, the terminal condition~\eqref{e.terminal_cond} is also satisfied. 

We put the $L^1$-metric on $\cM$ and put the supremum norm on continuous functions on $[0,1]\times \R^\D$. Then, \cite[Lemma~2.10]{chen2023Parisi} verifies that $\cM_\d \ni \alpha \mapsto \Phi_\alpha$ is Lipschitz. Hence, for any $\alpha\in\cM$, we can define $\Phi_\alpha$ to be the limit of $(\Phi_{\alpha_n})_{n\in\N}$ for $(\alpha_n)_{n\in\N}$ in $\cM_\d$ converging to $\alpha$. This is exactly the construction in \cite[Definition~2.11]{chen2023Parisi}, which is well-defined according to \cite[Corollary~2.12]{chen2023Parisi}. In particular, if $\alpha$ is continuous, then \cite[Proposition~2.17~(4)]{chen2023Parisi} shows that $\Phi_\alpha$ satisfies \eqref{e.Parisi_PDE} in the classical sense. 

We need the Lipschitzness of $\alpha\mapsto \Phi_\alpha$ given in \cite[Corollary~2.12]{chen2023Parisi}: there is a constant $C>0$ depending only on $\xi$ and $\Psi$ such that
\begin{align}\label{e.lip_Psi}
    \sup_{[0,1]\times \R^\D } \Ll|\Phi_\alpha(s,x) - \Phi_{\alpha'}(s,x)\Rr|\leq C\int_0^1\Ll|\alpha(s)-\alpha'(s)\Rr|\d s,\quad\forall \alpha,\alpha'\in \cM.
\end{align}

The condition~\eqref{e.mu_cond} is the same as in \cite[(2.6)]{chen2023Parisi}.
By \cite[Remark~2.2]{chen2023Parisi}, all results recalled above hold under~\eqref{e.mu_cond}.

If $\xi$ satisfies~\eqref{i.t.main_1} in Theorem~\ref{t.main}, then \cite[Lemma~4.2~(2)]{chen2023Parisi} implies~\eqref{e.mu_cond}. If $\xi$ satisfies~\eqref{i.t.main_2}, then \cite[Lemma~5.3~(2)]{chen2023Parisi} yields~\eqref{e.mu_cond}.
Hence, under the assumption on $\xi$ in Theorem~\ref{t.main}, the Parisi PDE solution $\Psi_\alpha$ exists and $\sF$ in \eqref{e.sF} is well-defined.

To prove Theorem~\ref{t.main}, we state \cite[Lemma~4.1]{chen2023Parisi} below, which holds due to~\eqref{e.mu_cond}.
The left-continuous inverse of $\alpha\in\cM$ is given by
\begin{align*}\alpha^{-1}(s) = \inf\{t\in[0,1]:s\leq \alpha(t)\},\quad\forall s\in [0,1].
\end{align*}
\begin{lemma}[\cite{chen2023Parisi}]\label{l.sF=sP}
Suppose that $\xi$ satisfies~\eqref{i.t.main_1} or~\eqref{i.t.main_2} in Theorem~\ref{t.main}.
If $\alpha\in\cM_\d$, then $\sF(\alpha)=\sP(\Psi\circ\alpha^{-1})$.
\end{lemma}

\begin{proof}[Proof of Theorem~\ref{t.main}]
Assuming that $\xi$ is convex on $\S^\D_+$, we first show~\eqref{e.limF_N=infF}.
In view of~\eqref{e.lip_psi_f} and~\eqref{e.sP=-psi+...}, we have that $\zeta \mapsto \sP(\Psi\circ\zeta)$ is Lipschitz over all increasing $\zeta$ satisfying $0\leq \zeta\leq1$.
Hence, for any $\eps>0$, we can find an increasing left-continuous step function $\zeta_\eps$ such that $\inf_\zeta \sP(\Psi\circ\zeta)\approx_\eps \sP(\Psi\circ \zeta_\eps)$, where $\approx_\eps$ is defined in~\eqref{e.approx_eps}. Choosing $\alpha_\eps\in\cM_\d$ to satisfy $\alpha_\eps^{-1} = \zeta_\eps$, we apply Lemma~\ref{l.sF=sP} to deduce $\inf_\zeta\sP(\Psi\circ\zeta)\geq \inf_\alpha\sF(\alpha)$. 

For the matching bound, we use~\eqref{e.lip_Psi} and the regularity of $\xi$ in~\ref{i.xi_loc_lip} to see that $\alpha\mapsto \sF(\alpha)$ is Lipschitz. Hence, for any $\eps>0$, we can find $\alpha_\eps\in\cM_\d$ such that $\inf_\alpha \sF(\alpha)\approx_\eps \sF(\alpha_\eps)$. Then, setting $\zeta_\eps=\alpha_\eps^{-1}$ and using Lemma~\ref{l.sF=sP}, we can get the matching bound and complete the proof of~\eqref{e.limF_N=infF}.

Next, we verify the strict convexity and the uniqueness of minimizers. If $\xi$ satisfies~\eqref{i.t.main_1} which is exactly the condition in \cite[Theorem~1.2]{chen2023Parisi}, then we can apply the theorem to get the announced results. If $\xi$ satisfies~\eqref{i.t.main_2}, then we can appeal to \cite[Theorem~5.1]{chen2023Parisi}.
\end{proof}

\small
\bibliographystyle{abbrv}
\newcommand{\noop}[1]{} \def\cprime{$'$}

\end{document}